\documentclass[reqno]{amsart}

    \usepackage{float}
    \usepackage{amsmath}
    \usepackage{graphicx}
    \usepackage{latexsym}
    \usepackage{amsfonts}
    \usepackage{amssymb}
    \usepackage{verbatim}
    \usepackage{color}
    \theoremstyle{plain}
\newtheorem{theorem}{Theorem}
\newtheorem{lemma}{Lemma}
\newtheorem{proposition}{Proposition}
\newtheorem{corollary}{Corollary}

\theoremstyle{definition}

\newtheorem{remark}{Remark}

\newtheorem{applemma}{Lemma}

\newtheorem*{remark*}{Remark}

    \makeatletter
    \@namedef{subjclassname@2020}{2020 Mathematics Subject Classification}
    \makeatother

\begin{document}
    \title[Diffusion approximation of a sequence of critical branching processes ]
    {Diffusion approximation of a sequence of critical branching processes with dependent immigration}
    %\thanks{Funded by the Deutsche Forschungsgemeinschaft (DFG, German Research Foundation) – Project-ID 317210226 – SFB 1283}
    %\subtitle{(Dedicated)}
    \author[Sharipov]{Sadillo Sharipov}
    \address{V.I.Romanovskiy Institute of Mathematics, Uzbekistan Academy of Sciences, Tashkent, Uzbekistan}
    \email{sadi.sharipov@yahoo.com}

\begin{abstract}
This paper aims to investigate a diffusion approximation of a sequence of critical branching processes with strictly stationary and ergodic immigration. Under fairly general assumptions on the dependence structure of the
immigration sequence, we establish that the scaled and properly normalized sequence
of branching processes with immigration converges in distribution to the diffusion process in the
Skorokhod topology. Our result extends the functional limit theorem of Wei and Winnicki (1989) to the case where the immigration sequence is dependent.
\end{abstract}

    % \version\vspace{1cm}

    \maketitle

%%%%%%%%%%%%%%%%%%%%%%%%%%%%%%%%%%%%%%%%%%%%%%%%%%%%%%%%%%%%%%%%%%%%%%%%%%%%%%%%%%%%%%%%%%%%%%%%%%%%%%%%%%
\section{Introduction}
Let $\left(\Omega, \mathfrak{F}, \mathbf{P} \right)$ be a probability space. The random variables with which we deal are all defined on the same probability space. Let $\left\{{{\xi_{k,i}},k,i \geq 1} \right\}$ be a sequence of independent and identically
distributed (i.i.d.), non-negative and integer-valued random variables. Let $\left\{ {{\varepsilon_k},k \geq 1} \right\}$ be another sequence of non-negative and integer-valued random variables. We assume that these two sequences are mutually independent.

A sequence of branching processes with immigration $\left\{X_{k},k \geq 0\right\}$ is defined by the following recursion
\begin{equation}\label{eq1}
X_{k}=\sum\limits_{i=1}^{X_{k-1}} {{\xi_{k,i}}+{\varepsilon_k}},\ \  k \geq 1,			
\end{equation}
where we set $\sum\limits_{i=1}^{0}:=0$. Initial value $X_{0}$ is assumed a non-negative, integer-valued and square-integrable random variable that is independent of $\left\{ {{\xi_{k,i}}} \right\}$ and $\left\{ {{\varepsilon_k}} \right\}$.
Note that the independence assumption for the sequences $\{\xi_{k,i}\}$ and $\{\varepsilon_{k}\}$ means that reproduction and immigration processes are mutually independent.

Intuitively, one can interpret $\xi_{k,i}$ as the number of offsprings produced by the $i$-th individual belonging to the $(k-1)$-th generation, and $\varepsilon_k$ is the number of immigrants in the $k$-th generation. We can interpret $X_{k}$ as the number of individuals in the $k$-th generation.
Denote
$$ a:=\mathbf{E}(\xi_{1,1}), \ \   \sigma^{2}:=\operatorname{Var}\left({\xi_{1,1}}\right), \ \ \lambda:=\mathbf{E}(\varepsilon_{1}),\ \ b^{2}:=\operatorname{Var}\left({\varepsilon_{1}}\right), \ \ \rho\left(k\right):=\operatorname {Cov}\left(\varepsilon_{1}, \varepsilon_{1+k}  \right). $$

A sequence $\left\{X_{k}\right\}$ is said to be subcritical, critical, or supercritical depending on $a<1$, $a=1$, or $a>1$, respectively.
Throughout the paper, we consider the critical case.

The asymptotic behavior of the process defined by \eqref{eq1} has been investigated in a series of papers. For instance, when the immigration is an i.i.d., Seneta \cite{S70} established convergence in distribution of $n^{-1}X_{n}$ towards a gamma-distribution. Extensive research than has been performed to extend this result in various directions.
For a comprehensive discussion, we refer the reader to the recent review paper by Rahimov \cite{R21}. We focus our attention on the case where the immigration sequence is generated by dependent random variables. In \cite{N75}, Nagaev relaxed the i.i.d. assumption on the immigration sequence to the case of wide-sense stationary immigration and proved that Seneta's result remains valid. Later, Asadullin and Nagaev \cite{AN82} considered both discrete and continuous time branching processes with immigration and showed that Nagaev's result \cite{AN82} still remains true under the more general condition that there exists a random variable $\varepsilon$, such that
${n^{-1}}\mathbf{E}\left|{\sum\limits_{k=1}^n {\left({{\varepsilon_k}-\varepsilon }\right)}}\right| \to 0$  as $n \to \infty$. It is worth noting that neither dependence structure nor restriction on the distribution of immigration sequence was imposed in \cite{AN82}. Badalbaev and Zubkov \cite{BZ83} established a limit theorem for a sequence of special random processes, including branching processes with immigration. Their theorem contains the results of \cite{N75} and \cite{AN82} as special cases. Note that the convergence of finite-dimensional distributions of the sequence of processes (1)
was studied by Kawazu and Watanabe \cite{KW71} and Aliev \cite{A85}. We would like also to stress that Guo and Zhang \cite{GZ14} extended one of the functional limit theorems of \cite{R07} to the case where there exists $N\geq 2$ such that $\varepsilon_{i}$ and $\varepsilon_{j}$ are independent whenever $\left|i-j \right|> N$. This result was extended in \cite{S24} to the case where the immigration sequence satisfies $\rho$-mixing condition. We also refer the reader to the papers \cite{GS20}, \cite{Kh17}, \cite{LZ20}, \cite{RS24} and \cite{S23}, where functional limit theorems were proved under various dependence condition on immigration process.

However, all of the above mentioned papers do not consider the situation when scaled and properly normalized process \eqref{eq1} converge in distribution to the diffusion process under the assumption that the immigration sequence is weakly dependent. The purpose of this paper is to fill this gap by establishing a functional limit theorem for the critical process \eqref{eq1}, under the assumption that the immigration sequence is strictly stationary, ergodic and satisfies a dependence condition that covers martingale sequences, mixingales and various mixing conditions. Our motivation comes from the celebrated result of Wei and Winnicki \cite{WW89}, who established that the scaled and properly normalized process \eqref{eq1} converges weakly in the Skorokhod topology to a non-negative diffusion process, provided that the immigration sequence $\{\varepsilon_k\}$ is i.i.d. %It is natural to ask whether the result of Wei and Winnicki remains valid under weak dependence assumptions on the immigration sequence.
It is worth noting that the method used in Wei and Winnicki's proof is based on operator semigroup convergence theorems and strongly relies on the Markov property of the process \eqref{eq1}. This method was applied by Ethier and Kurtz \cite[Theorem~9.1.3]{EK86} to derive a functional limit theorem for branching processes without immigration. However, this method is not appropriate in the context of dependent immigration, since process \eqref{eq1} is no longer a Markov chain for the dependent immigration sequence. Therefore, in the proofs we follow the method introduced by Isp\'any and Pap \cite{IP10}. Namely, the basic idea is to apply a limit theorem for random step processes towards a diffusion process. This method was used in series of papers \cite{IP14}, \cite{BBP21}, \cite{GMP23} and \cite{BGMP24} to derive some fluctuation limit theorems for process \eqref{eq1} and some models of branching processes as well. A detailed alternative proof was demonstrated by Barczy et al. \cite{BBP21}, who derived the result of Wei and Winnicki using limit theorems for random step processes
toward a diffusion process. Their proof serves as a useful basis for our analysis.

The paper is organized as follows. In Section 2 we give the main result. Section 3
contains some preliminary lemmas and their proofs. Section 4 contains the proofs. For ease of reading, we include in the Appendix a functional limit theorem for sequences of martingale differences due to Isp\'any and Pap \cite{IP10}. Furthermore, we recall a version of the continuous mapping theorem and maximal $\mathbf{L}_{p}$-maximal inequality for strictly stationary process. %For the ease of reading the paper, preliminary results are presented in the Appendix.

%%%%%%%%%%%%%%%%%%%%%%%%%%%%%%%%%%%%%%%%%%%%%%%%%%%%%%%%%%%%%%%%%%%%%%%%%%%%%%%%%%%%%%%%%%%%%%%%%%%%%%%%%%
\section{Fluctuation limit theorem}
In order to present our results, let us first introduce dependence condition used throughout the paper.

To formulate dependence assumptions for the immigration sequence, it is convenient to consider a
two-sided strictly stationary ergodic extension of $\left\{\varepsilon_{k}, k\ge 1\right\}$. More
precisely, we assume that there exists a strictly stationary ergodic sequence
$ \left\{\widetilde{\varepsilon}_{k}, k \in \mathbb{Z}\right\}$
such that
$$
(\widetilde{\varepsilon}_1, \widetilde{\varepsilon}_2,\dots)
\stackrel{\mathcal{D}}{=}
(\varepsilon_1,\varepsilon_2,\dots),
$$
where $\stackrel{\mathcal{D}}{=}$ denotes equality in distribution.

We then define the past sigma-algebras
$$
\mathcal{H}_{k}:=\sigma(\widetilde{\varepsilon}_{j}, j\le k),
\ \ k \in \mathbb{Z}.
$$

Thus, the two-sided extension is used only for the immigration sequence in order to
define stationary past sigma-algebras such as $\mathcal{H}_{0}$ and to impose dependence
conditions involving conditional expectations with respect to $\mathcal{H}_{0}$.
%It does
%not alter the one-sided branching process itself, and the filtration \((\mathcal G_k)\)
%remains the filtration used in the proof of the functional limit theorem.

The stationary process considered throughout the paper is rigorously defined as follows.
Let $\mathbb{T}: \Omega \mapsto \Omega$ be a bijective bi-measurable transformation preserving the probability $\mathbf{P}$. Let $\mathcal H_{0}$ be a $\sigma$-algebra of $\mathfrak{F}$ satisfying $\mathcal H_{0} \subseteq \mathbb{T}^{-1}\left(\mathcal{H}_0\right)$. We define the nondecreasing filtration $\left\{\mathcal{H}_{i}, i \in \mathbb{Z}\right\}$ by $\mathcal{H}_{i}=\mathbb{T}^{-i}\left(\mathcal{H}_0\right)$ and the stationary sequence $\left\{\varepsilon_{i}, i \in \mathbb{Z}\right\}$ by $\varepsilon_{i}=\varepsilon_0 \circ \mathbb{T}^{i}$, where $\varepsilon_{0}$ is a real-valued random variable. The sequence will be called adapted to the filtration $\left\{\mathcal H_{i}, i \in \mathbb{Z}\right\}$ if $\varepsilon_{0}$ is $\mathcal H_{0}$-measurable.
Denote $\zeta_{i}:=\varepsilon_{i}-\lambda$, $i\geq 1$ and set $S_{n}=\sum_{i=1}^{n}\zeta_{i}$.

Consider the following two quantities:
$$ \delta_{n,2}=\sum_{k=1}^{n}k^{-3/2}\left\| \mathbf{E}\left(S_{k}| \mathcal H_{0} \right) \right\|_{2}, \ \ \delta_{\infty,2}=\sum_{k=1}^{\infty}k^{-3/2}\| \mathbf{E}\left(S_{k}| \mathcal H_{0} \right) \|_{2}. $$
Maxwell and Woodroofe \cite{MW00} showed that the condition $\delta_{\infty,2}<\infty$ guarantees the central limit theorem. In the literature, condition $\delta_{\infty, 2}<\infty$ is referred as Maxwell-Woodroofe condition. Later, Peligrad and Utev \cite{PU05} obtained an $\mathbf{L}_{2}$-maximal inequality and used it to prove the invariance principle for partial sums of strictly stationary processes.
%In the rest of this section the sequence: $\left(X_i\right)_{i \in Z}$ is assumed to be stationary and adapted to $\left(\mathcal{F}_i\right)_{i \in Z}$ and the variables are $\ln \mathbf{L}^p$.

For each $n \geq 1$, we introduce the random process $Y_{n}=\left\{Y_{n}\left(t\right), t\geq 0\right\}$, defined by
$$ Y_{n}\left(t\right):=\frac{X_{[nt]}}{n},\ \ t \geq 0, $$
where $\left[\cdot \right]$ denotes the integer part.
Throughout the paper, the symbols $\mathop{\to}\limits^{\mathbf{P}}$, $\mathop{\to}\limits^{\mathbf{L^{2}}}$, $\mathop{\to}\limits^{{d}}$ and $\mathop{\to}\limits^{\mathcal{D}}$ denote convergence in probability, convergence in mean square, convergence in distribution and the weak convergence of random functions in the space $D[0,\infty)$ with the Skorokhod topology, respectively. In the sequel, $\|\cdot \|_{p}$ denotes the usual $\mathbf{L}_{p}$ norm.
%In this paper, we shall study the fluctuation limit of $\left\{X_{k}, k \geq 0\right\}$ defined by (1) under the following condition:

%(C): The immigration sequence $\left\{\varepsilon_{k}, k \geq 1 \right\}$ is a stationary in wide sense, i.e. $\mathbb{E}\varepsilon_{1}=\lambda$, and covariance function $\rho\left(k\right):=\operatorname {cov}\left(\varepsilon_{i+k}, \varepsilon_{i}  \right)$ depends only on $i-k$.

We are ready to formulate our main result.
\begin{theorem}\label{thm1}
Let $\left\{X_{k}, k \geq 0 \right\}$ be a critical branching process with immigration with $\lambda \in \left(0, \infty\right)$, $b^{2}\in \left(0, \infty\right)$ and $\sigma^{2}\in \left(0,\infty\right)$.
Assume that the immigration sequence $\left\{{{\varepsilon_k},k\geq 1} \right\}$ is a strictly stationary and ergodic sequence adapted to a stationary non-decreasing filtration $\left\{\mathcal{H}_{i}, i \in \mathbb{Z}\right\}$ such that $\delta_{n,2}=o\left(\sqrt{n}\right)$ as $n \to \infty$.

Then
\begin{equation}\label{eq1*}
 Y_{n} \mathop{\to}\limits^{\mathcal{D}} Y, \qquad n\to \infty,
\end{equation}
where the limit process $Y=\left\{Y\left(t\right), t \geq 0 \right\}$ is the pathwise unique strong solution of the stochastic differential equation (SDE)
\begin{equation}\label{eq2}
 \mathrm{d} Y\left(t\right) = \lambda \mathrm{d}t+\sigma \sqrt{Y\left(t\right)^{+}} \mathrm{d} W\left(t\right)
\end{equation}
with initial condition $Y(0)=0$, where $\left\{W\left(t\right), t \geq 0 \right\}$ is a standard Wiener process.
\end{theorem}
The limit process $\left\{Y(t), t \geq 0\right\}$ is called a squared Bessel process. It is known (see \cite{IW89}) that the SDE \eqref{eq2} admits a pathwise unique solution $\{Y^{(x)}(t), t \geq 0\}$
for each initial condition $Y^{(x)}(0)=x \in \mathbb{R}$. Furthermore, for
any $x\geq 0$, the corresponding solution
remains nonnegative almost surely, that is,
$Y^{(x)}(t)\geq 0$ a.s. for all $t\geq 0$. Thus, we can replace $Y(t)^{+}$ by $Y(t)$ under the square root in \eqref{eq2}.
\begin{remark}
Note that the statement of Theorem~\ref{thm1} remains valid if we replace condition $\delta_{n,2}=o(\sqrt{n})$ as $n \to \infty$ by the Maxwell-Woodroofe condition. This condition holds for a large class of stationary processes with dependence structure. For instance, according to Lemma 1 in \cite{PUW07}, the quantity $\delta_{\infty,2}$ is finite if the immigration sequence satisfies the so-called $\rho$-mixing condition with logarithmic decay of correlation. Moreover, the finiteness of $ \delta_{\infty,2}$ follows from the fact that the immigration sequence is a mixingale.
We also note that $\delta_{n,2}=0$ whenever the immigration sequence is a stationary martingale difference sequence.
\end{remark}
\begin{remark}
In the proof, we deal with the filtration
$$
\mathfrak{F}_{k}
:=
\sigma\left\{X_{0},(\xi_{\ell,i},\varepsilon_{\ell}), 1\le \ell\le k,\ i\ge 1\right\},
\qquad k\ge 1,
$$
rather than with the natural filtration of the branching process
$$
\mathcal{F}_{k}
:=
\sigma\left\{X_{0}, X_{1},\dots, X_{k}\right\},
\qquad k \ge 0.
$$
This choice is motivated by the dependent structure of immigration sequence. Indeed, since the
offspring family $\{\xi_{k,i}\}$ is independent of the immigration sequence
$\{\varepsilon_{k}\}$, we obtain
$$
\mathbf{E}\left(\varepsilon_{k}\mid \mathfrak{F}_{k-1}\right)
=
\mathbf{E}\left(\varepsilon_{k}\mid \sigma(\varepsilon_{1},\dots, \varepsilon_{k-1})\right),
\ \ k \ge 1.
$$
Hence, the predictor of $\varepsilon_{k}$ with respect to filtration $\mathfrak{F}_k$ is determined
only by the past of the immigration process which preserves the stationary and ergodic
structure needed in the proof.

By contrast, under the natural filtration $\mathcal{F}_{k}$, one has to deal with random variable
$$
\mathbf{E}(\varepsilon_{k}\mid \mathcal F_{k-1})
=
\mathbf{E}(\varepsilon_{k}\mid X_0,\dots, X_{k-1}),
$$
which is much less tractable since it is not
stationary. Therefore, the sequence
$
\left\{\mathbf{E}(\varepsilon_{k}\mid \mathcal{F}_{k-1}), k \ge 1\right\}
$
does not, in general, admit a transparent stationary structure.
We also note that, in general,
$$
\mathcal{F}_{k}\subsetneq \mathfrak{F}_{k}.
$$
Indeed, each $X_{j}$, $0\le j\le k$, is measurable with respect to $\mathfrak{F}_{k}$, and
hence $\mathcal{F}_{k}\subset \mathfrak{F}_{k}$. However, the reverse inclusion fails,
because $\mathfrak{F}_{k}$ contains the individual variables, whereas
$\mathcal{F}_{k}$ contains only the observed population sizes.
\end{remark}

%For example, if \(X_0=1\), then
%\[
%X_1=\xi_{1,1}+\varepsilon_{1},
%\]
%hence
%%\mathcal{F}_{1}^{X}=\sigma(X_1)=\sigma(\xi_{1,1}+\varepsilon_1),
%\]
%whereas
%\[
%\mathcal{G}_{1}
%=
%\sigma(\xi_{1,1},\varepsilon_1,\xi_{1,2},\xi_{1,3},\dots).
%\]
%Thus $\mathcal{F}_{1}^{X}$ contains only the sum \(\xi_{1,1}+\varepsilon_1\), while
%$\mathcal{G}_{1}$ contains the individual %innovation variables, so these filtrations do not
%coincide in general.
The total progeny is one of the important functionals associated with the process.
More generally, we consider sums of the form
$$
\sum_{k=0}^{n} f\!\left(Y_{n}\!\left(\frac{k}{n}\right)\right),
$$
where $f:[0,\infty)\to\mathbb{R}$ is a continuous function. Such functionals arise naturally in the derivation of asymptotic distributions of certain estimators.

Observe that
$$
\frac{1}{n}\sum_{k=0}^{n} f\!\left(Y_n\!\left(\frac{k}{n}\right)\right)
=
\frac{1}{n}f(Y_n(0))
+
\sum_{k=1}^{n}\int_{(k-1)/n}^{k/n}
f\!\left(Y_n\!\left(\frac{k}{n}\right)\right)\,dt.
$$
Define the functionals $\Psi_{n}:D[0,\infty) \to \mathbb{R}$ by
$$
\Psi_{n}(x)
:=
\frac{1}{n}\sum_{k=1}^{n} f\!\left(x\!\left(\frac{k}{n}\right)\right),
\qquad x \in D[0,\infty),
$$
and set
$$
\Psi(x):=\int_{0}^{1} f(x(t))\,\mathrm{d}t,
\qquad x \in D[0,\infty).
$$

Let $x_{n}, x \in D[0,\infty)$ be such that
$
\sup_{t\in[0,1]}|x_{n}(t)-x(t)| \to 0
$
and $x$ is continuous. Then
$$
|\Psi_{n}(x_{n})-\Psi(x)|\to 0,
\qquad n\to\infty.
$$
Indeed, one can verify that
$$
|\Psi_{n}(x_{n})-\Psi(x)|$$
$$ \le
\frac{1}{n}\sum_{k=1}^{n}
\left|
f\!\left(x_{n}\!\left(\frac{k}{n}\right)\right)
-
f\!\left(x\!\left(\frac{k}{n}\right)\right)
\right|
+
\left|
\frac{1}{n}\sum_{k=1}^{n} f\!\left(x\!\left(\frac{k}{n}\right)\right)
-
\int_{0}^{1} f(x(t))\,\mathrm{d}t
\right|.
$$
Since $x$ is continuous on $[0,1]$, then it is bounded and for all sufficiently large $n$,
the functions $x_n$ are uniformly bounded on $[0,1]$. Hence, by continuity of $f$, the first
term of above inequality vanishes, while the second one tends to zero by the Riemann sum theorem.

Due to the fact that the limit process $Y$ is driven by a Wiener process, it has continuous sample paths almost surely (a.s.). Hence,
by the extended continuous mapping theorem; see Lemma A.1 from Appendix,
$$
\Psi_{n}(Y_{n}) \xrightarrow[]{\mathcal{D}} \Psi(Y),
\qquad n \to \infty.
$$
Moreover,
$$
\frac{1}{n} f(Y_{n}(0))\xrightarrow[]{\mathbf{P}}0,
\qquad n\to\infty.
$$
Consequently, we have proved the following corollary.
\begin{corollary}\label{cor1}
Under the assumptions of Theorem~\ref{thm1}, for every continuous function
$f:[0,\infty)\to\mathbb R$,
$$
\frac{1}{n}\sum_{k=0}^{n} f\!\left(Y_n\!\left(\frac{k}{n}\right)\right)
\xrightarrow[]{\mathcal{D}}
\int_{0}^{1} f(Y(t))\,\mathrm{d}t,
\qquad n\to\infty.
$$
\end{corollary}

In particular, taking $f(x)=x^{\theta}$, $\theta \geq 0$, in Corollary~1, we obtain the following result.
\begin{corollary}\label{cor2}
Under the conditions of Theorem~\ref{thm1}, we have for each
$\theta \geq 1$,
$$
\frac{1}{n^{\theta+1}}\sum_{i=0}^{n}X_{i}^{\theta} \xrightarrow[]{\mathcal{D}} \int_{0}^{1} Y^{\theta}(t)\,\mathrm{d}t, \ \ n \to \infty.
$$
\end{corollary}
 %%%%%%%%%%%%%%%%%%%%%%%%%%%%%%%%%%%%%%%%%%%%%%%%%%%%%%%%%%%%%%%%%%%%%%%%%%%%%%%%%%%%%%%%%%%%%%%%%%%%%%%%%%

\section{Auxiliary results}

For each $k\geq 0$, let us define the filtration $\mathfrak{F}_{k}:=\sigma \left\{X_{0},\,(\xi_{\ell,i},\varepsilon_{\ell}): \ell\le k,\ i\ge 1\right\}$ with $\mathfrak{F}_{0}:=\sigma \left\{X_{0}\right\}$. For $k \geq 1$, we denote
$$
\eta_{k}:=\mathbf{E}(\varepsilon_{k}\mid \mathfrak F_{k-1}), \ \ \nu_{k}:=\operatorname{Var}(\varepsilon_{k}\mid \mathfrak{F}_{k-1}).
$$
It is clear that the sequence $\left\{M_{k},k\ge 1\right\}$ for each $k\geq 1$, defined as
\begin{equation}\label{eq2*}
 M_{k}:=X_{k}-{\mathbf{E}} \left(X_{k} \left|\mathfrak{F}_{k-1} \right. \right)=X_{k}- X_{k-1}-\eta_{k}
\end{equation}
is a martingale difference sequence with respect to the filtration $\left\{\mathfrak{F}_{k}, k\geq 0\right\}$.

Thus,
$$ M_{k}=T_{k}+N_{k}, $$
where
\begin{equation}\label{eq3}
T_{k}: =\sum_{j=1}^{X_{k-1}}\left(\xi_{k,j}-1 \right), \ \  N_{k}:=\varepsilon_{k}-\eta_{k}.
\end{equation}
%Hence, $M_{n}$ can be represented as follows:
%\begin{equation}\label{eq7}
%M_{n}=\eta_{n}+\zeta_{n},
%\end{equation}
%where
%$$
%\eta_{n}=\sum_{j=1}^{X_{n-1}}\left(\xi_{n,j}-1\right), \ \  \zeta_{n}=\varepsilon_{n}-{\mathbf{E}} \left(\varepsilon_{n} \left|\Im_{n-1} \right. \right).
%$$
In view of \eqref{eq3}, we obtain
\begin{equation}\label{eq4}
 X_{n}=X_{0}+\sum_{k=1}^{n}M_{k}-\sum_{k=1}^{n}N_{k}+\sum_{k=1}^{n}\zeta_{k}.
\end{equation}
Taking into account \eqref{eq4}, we rewrite $Y_{n} \left(t\right)$ as
$$ Y_{n} \left(t\right)=Y_{n}^{\left(1\right)} \left(t\right)-Y_{n}^{\left(2\right)} \left(t\right)+Y_{n}^{\left(3\right)} \left(t\right),$$
where
$$ Y_{n}^{\left(1\right)}\left(t\right):=\frac{1}{n} \left(X_{0}+\sum_{k=1}^{\left[nt\right]}M_{k}\right), \ \   Y_{n}^{\left(2\right)} \left(t\right):=\frac{1}{n} \sum_{k=1}^{\left[nt\right]}N_{k}, \ \
 Y_{n}^{\left(3\right)}\left(t\right):=\frac{1}{n} \sum_{k=1}^{\left[nt\right]}\zeta_{k}.$$
The following lemma plays an important role in the subsequent arguments and is also of independent interest.
\begin{lemma}\label{lem1}
Let $\left\{\varepsilon_{k},k\in\mathbb Z\right\}$ be a sequence of strictly stationary and ergodic random variables that is independent of $\left\{\xi_{k,i}\right\}$. Then the sequence $\left\{\eta_k, k\in\mathbb Z\right\}$ is also strictly stationary and ergodic.
\end{lemma}
\begin{proof}
For each $k\in\mathbb Z$, denote
$\mathcal{H}_{k}:=\sigma\left\{\varepsilon_\ell, \ell \le k \right\}$.
Due to independence assumption of $\left\{\xi_{k,i}\right\}$ from $\left\{\varepsilon_{k}\right\}$, we get
$$
\mathbf{E}(\varepsilon_k\mid \mathfrak{F}_{k-1})
=
\mathbf{E}(\varepsilon_k\mid \mathcal{H}_{k-1}),
\qquad k\in\mathbb Z
$$
which means that conditioning on the past offspring variables gives no additional information about
$\varepsilon_{k}$. Hence,
$$
\eta_{k}=\mathbf{E}(\varepsilon_{k} \mid \mathcal{H}_{k-1}), \qquad k \in \mathbb Z.
$$
Since $\left\{\varepsilon_{k}, k\in\mathbb Z\right\}$ is strictly stationary, then there exists a
measure-preserving transformation $\mathbb{T}$ on the underlying probability space such that
$
\varepsilon_{k}=\varepsilon_{0}\circ \mathbb{T}^k$, $k\in\mathbb Z$.
Moreover, it follows that
$
\mathcal{H}_{k-1}=\mathbb{T}^{-k}\mathcal{H}_{-1}$,
$k\in\mathbb Z$.

Indeed, one can verify that
$$
\mathbb{T}^{-k}\mathcal{H}_{-1}
=
\mathbb{T}^{-k}\sigma \left\{\varepsilon_{\ell}, \ell \le -1\right\}
=
\sigma \left\{\varepsilon_{\ell+k}, \ell \le -1\right\}
=
\sigma \left\{\varepsilon_{m}, m\le k-1\right\}
=
\mathcal{H}_{k-1}.
$$
Now, let us set $\eta_{0}:=\mathbf{E}(\varepsilon_{0} \mid \mathcal{H}_{-1})$.
Using the property of conditional expectation under measure-preserving
transformations, namely,
$$
\mathbf{E}(Z\circ \mathbb{T}^k\mid \mathbb{T}^{-k}\mathcal A)
=
\mathbf{E}(Z\mid \mathcal A)\circ \mathbb{T}^k,
$$
we obtain for each $k\in\mathbb Z$ that
$$
\eta_k
=
\mathbf{E}(\varepsilon_k\mid \mathcal{H}_{k-1})
=
\mathbf{E}(\varepsilon_{0}\circ \mathbb{T}^{k}\mid \mathbb{T}^{-k}\mathcal{H}_{-1})
=
\mathbf{E}(\varepsilon_{0}\mid \mathcal{H}_{-1})\circ \mathbb{T}^{k}
=
\eta_{0}\circ \mathbb{T}^{k}.
$$
Therefore, the sequence $\left\{\eta_{k}, k\in\mathbb Z\right\}$ is strictly stationary.

It remains to prove ergodicity. Due to the measurability of $\eta_{k}$ with respect to
$\sigma\left\{\varepsilon_{\ell}, \ell \in \mathbb Z\right\}$, we have
$$
\sigma\left\{\eta_{k}, k\in\mathbb Z\right\}\subset \sigma\left\{\varepsilon_{\ell}, \ell \in \mathbb Z\right\}.
$$
Let $A\in \sigma\left\{\eta_k, k\in\mathbb Z\right\}$ be invariant set under the shift $\mathbb{T}$, i.e.
$\mathbb{T}^{-1}A=A$. Then $A$ is also an invariant set for the process
$\left\{\varepsilon_{k}, k\in\mathbb Z\right\}$. Since $\left\{\varepsilon_k, k\in\mathbb Z\right\}$ is ergodic,
it follows that
$
\mathbf{P}(A)\in\{0,1\}.
$
Hence, the sequence $\left\{\eta_k, k\in\mathbb Z\right\}$ is ergodic.
This completes the proof.
\end{proof}

For each $t\ge 0$, $n\geq 1$, consider the process $A_{n}=\left\{A_{n}(t), t \geq 0 \right\}$ defined as
$$
A_{n}(t):=\frac{1}{n}\sum_{k=1}^{[ nt]}\eta_{k}.
$$
\begin{lemma}\label{lem2}
For each fixed $T>0$, as $n\to \infty$,
\begin{equation}\label{eq5}
\sup_{t\in [0,T]} \left|A_{n}(t)-\lambda t\right| \to 0
\qquad\text{a.s.}
\end{equation}
\end{lemma}
\begin{proof}
By Lemma \ref{lem1}, the sequence $\left\{\eta_{k}, k \geq 1\right\}$ is stationary ergodic and integrable, with mean $\lambda$. Hence, Birkhoff's Ergodic Theorem yields that as $n\to \infty$,
$$
\frac{1}{n}\sum_{k=1}^{n}\left(\eta_{k}-\lambda\right) \to 0
\qquad\text{a.s.}
$$
Denote
$
\widetilde{S}_{j}:=\sum_{k=1}^{j}(\eta_k-\lambda)$, $j \ge1 $,  $\widetilde{S}_{0}:=0$.
Then, it clear that for each
$T>0$,
$$
\frac{1}{n} \max_{0\le j\le [nT]}\left|\widetilde{S}_{j}\right| \to 0
\qquad\text{a.s.}
$$
Therefore,
$$
\sup_{t\in [0,T]} \left|\frac{1}{n}\sum_{k=1}^{[nt]}(\eta_{k}-\lambda)\right|
=
\frac{1}{n} \max_{0\le j\le [nT]}\left|\widetilde{S}_{j}\right| \to 0 \qquad\text{a.s.}
$$
Next, observe that
$$
\sup_{t\in [0,T]}\left|\frac{[nt]}{n}-t\right|\le \frac{1}{n} \to 0, \ \ n \to \infty.
$$
In view of the triangle inequality,
$$
\sup_{t\in [0,T]}\left|A_{n}(t)-\lambda t\right|
\leq
\sup_{t\in [0,T]}\left|\frac{1}{n}\sum_{k=1}^{[nt]}(\eta_{k}-\lambda)\right|
+
\sup_{t\in [0,T]}\left|\lambda\left(\frac{[nt]}{n}-t\right)\right|,
$$
we arrive at the conclusion of lemma.
\end{proof}

In our proofs, we need following formulas for moments of branching processes with strictly stationary immigration.
\begin{lemma}\label{lem3}
Let $\left\{X_{k},k\geq 0 \right\}$ be the critical branching processes with strictly stationary immigration. If $\mathbf{E}X_{0}^{2}<\infty$, $\sigma^{2}<\infty$ and $b^{2}<\infty$, then for all $k\geq 1$, we have
\begin{equation}\label{eq6}
\mathbf{E}(X_k\mid \mathfrak F_{k-1})=X_{k-1}+\eta_{k},
\end{equation}
\begin{equation}\label{eq7}
\mathbf{E}(X_k)=\mathbf{E}(X_{k-1})+\lambda = \mathbf{E}(X_0)+\lambda k,
\end{equation}
\begin{equation}\label{eq8}
\operatorname{Var}(X_{k}\mid \mathfrak F_{k-1})
=\mathbf{E}(M_{k}^{2} \mid \mathfrak F_{k-1})
=\sigma^{2} X_{k-1}+\nu_{k},
\end{equation}
\begin{equation}\label{eq9}
\operatorname{Var}(X_{k})
=\operatorname{Var}(X_{k-1})
+\sigma^{2}\mathbf{E}(X_{k-1})
+\sigma^{2}
+2 \operatorname{Cov}(X_{k-1}, \eta_{k}),
\end{equation}
\begin{equation}\label{eq10}
\mathbf{E}(M_{k}^{2})
= \sigma^{2} \mathbf{E}(X_{k-1})
+\mathbf{E}\nu_{k}.
\end{equation}
\end{lemma}
\begin{proof}
Due to criticality of process,
$$
\mathbf{E}\!\left(\sum_{j=1}^{X_{k-1}}\xi_{k,j}\,\middle|\,\mathfrak F_{k-1}\right)
=
\sum_{j=1}^{X_{k-1}}\mathbf{E}(\xi_{k,j})
=
X_{k-1},
$$
which yields
$$
\mathbf{E}(X_k\mid \mathfrak F_{k-1})
=
\mathbf{E}\!\left(\sum_{j=1}^{X_{k-1}}\xi_{k,j}+\varepsilon_k\,\middle|\,\mathfrak{F}_{k-1}\right)
=
X_{k-1}+\eta_{k}.
$$
This proves \eqref{eq6}.

Further, observe that
$$
\mathbf{E}(X_k)=\mathbf{E}(X_{k-1})+\mathbf{E}(\eta_k)
=\mathbf{E}(X_{k-1})+\mathbf{E}(\varepsilon_k)
=\mathbf{E}(X_{k-1})+\lambda,
$$
and strict stationarity of $\left\{\varepsilon_k\right\}$ gives \eqref{eq7}.

It is easily checked that
$$
M_{k}
=
X_{k}-\mathbf{E}(X_{k}\mid \mathfrak{F}_{k-1})
=
\sum_{j=1}^{X_{k-1}}(\xi_{k,j}-1)+(\varepsilon_{k}-\eta_{k}).
$$
By the independence of the offspring family and the immigration sequence, the cross term has
conditional expectation zero given $\mathfrak{F}_{k-1}$. Therefore,
$$
\mathbf{E}(M_{k}^{2} \mid \mathfrak{F}_{k-1})
=
\mathbf{E}\!\left[\left(\sum_{j=1}^{X_{k-1}}(\xi_{k,j}-1)\right)^{2} \middle|\,\mathfrak{F}_{k-1}\right]
+
\nu_{k}.
$$
Since, conditionally on $\mathfrak{F}_{k-1}$, the variables $\xi_{k,j}-1$ are centered and
independent, we have
$$
\mathbf{E}\!\left[\left(\sum_{j=1}^{X_{k-1}}(\xi_{k,j}-1)\right)^{2} \middle|\,\mathfrak{F}_{k-1}\right]
=
\sigma^{2} X_{k-1}.
$$
Thus, \eqref{eq8} follows.

In order to prove \eqref{eq9}, we use the recursion
$$
X_{k}=X_{0}+\sum_{i=1}^{k} T_{i}+\sum_{i=1}^{k} \varepsilon_{i}, \ \ k \geq 1.
$$
This leads us to write
$$
\operatorname{Var}(X_{k})
=
\operatorname{Var}(X_{0})
+
\operatorname{Var}\!\left(\sum_{i=1}^{k} T_{i}\right)
+
\operatorname{Var}\!\left(\sum_{i=1}^{k} \varepsilon_{i}\right)$$
$$
+2\operatorname{Cov}\!\left(X_{0},\sum_{i=1}^{k} T_{i}\right)
+
2\operatorname{Cov}\!\left(X_{0},\sum_{i=1}^{k} \varepsilon_{i}\right)
+
2\operatorname{Cov}\!\left(\sum_{i=1}^{k} T_{i},\sum_{i=1}^{k} \varepsilon_{i}\right).
$$
We first show that all covariance terms vanish. Since $X_{0}$ is independent of the
immigration sequence, we get
$$
\operatorname{Cov}\!\left(X_{0},\sum_{i=1}^{k} \varepsilon_{i}\right) = 0.
$$
Next, observe that for each $i \geq 1$, $\mathbf{E}(T_{i}\mid \mathfrak{F}_{i-1})=0$,
and since $X_{0}$ is $\mathfrak{F}_{i-1}$-measurable, one has that
$$
\mathbf{E}(X_{0}T_{i})=\mathbf{E}\!\left[X_{0}\,\mathbf{E}(T_{i} \mid \mathfrak{F}_{i-1})\right]=0,
$$
therefore,
$$
\operatorname{Cov}\!\left(X_{0},\sum_{i=1}^{k} T_{i}\right)=0.
$$
Next, for every $i,j\geq 1$,
$$
\mathbf{E}(T_{i}\varepsilon_{j})
=
\mathbf{E}\!\left(\varepsilon_{j}\,\mathbf{E}(T_{i}\mid X_{i-1},\varepsilon_{j})\right).
$$
Due to the fact that the offspring variables $\{\xi_{i,r}, r\ge 1\}$ are independent of $(X_{i-1},\varepsilon_{j})$, we get
$$
\mathbf{E}(T_{i}\mid X_{i-1},\varepsilon_{j})=\mathbf{E}(T_{i} \mid X_{i-1})=0.
$$
Thus, $\mathbf{E}(T_{i}\varepsilon_{j})=0$, and consequently,
$$
\operatorname{Cov}\!\left(\sum_{i=1}^{k} T_{i},\sum_{i=1}^{k} \varepsilon_{i}\right)=0.
$$
Hence,
$$
\operatorname{Var}(X_{k})
=
\operatorname{Var}(X_{0})
+
\operatorname{Var}\!\left(\sum_{i=1}^{k} T_{i}\right)
+
\operatorname{Var}\!\left(\sum_{i=1}^{k} \varepsilon_{i}\right).
$$
Taking into account the fact that for $1\le i<j\le k$, the variable $T_i$ is $\mathfrak{F}_{j-1}$-measurable and noting
$\mathbf{E}(T_{j}\mid \mathfrak{F}_{j-1})=0$, we obtain
$$
\mathbf{E}(T_{i}T_{j})=\mathbf{E}\!\left[T_{i}\,\mathbf{E}(T_{j}\mid \mathfrak{F}_{j-1})\right]=0.
$$
Therefore, using $\mathbf E(T_{i}^{2})=\sigma^{2}\mathbf{E}(X_{i-1})$ and $\mathbf{E}(X_{i-1})=\mathbf{E}(X_{0})+\lambda(i-1)$,
we conclude
$$
\operatorname{Var}\!\left(\sum_{i=1}^{k} T_{i}\right)=\sum_{i=1}^{k} \mathbf{E}(T_{i}^{2})=
\sigma^2\sum_{i=1}^{k} \mathbf{E}(X_{i-1})
=
\sigma^{2} k\,\mathbf{E}(X_{0})
+
\frac{\lambda\sigma^{2}}{2}k(k-1).
$$
Consequently, by strict stationarity of $\left\{\varepsilon_k\right\}$,
$$
\operatorname{Var}\!\left(\sum_{i=1}^{k} \varepsilon_i\right)
=
\sum_{i=1}^{k} \operatorname{Var}(\varepsilon_i)
+
2\sum_{1\le i<j\le k}\operatorname{Cov}(\varepsilon_{i},\varepsilon_{j})
=
kb^{2}+2\sum_{j=1}^{k}(k-j)\rho(j).
$$
This ends the proof of \eqref{eq9}.

Note that \eqref{eq10} is consequence of \eqref{eq8}. Thus, Lemma is proved.
\end{proof}

\begin{lemma}\label{lem4}
Assume that the conditions of Lemma~3 hold. If $\delta_{n,2}=o\left(\sqrt{n}\right)$ as $n\to \infty$, then, as $n\to \infty$,
$$
\operatorname{Var}(X_{n})=O(n^{2}), \qquad
\mathbf{E}(X_{n}^{2})=O(n^{2}),\qquad
\mathbf{E}(|M_{n}|)=O(n^{1/2}),\qquad
\mathbf{E}(M_{n}^{2})=O(n).
$$
\end{lemma}
\begin{proof}
It is well known (see \cite[Lemma A.1]{IP14}) that if the immigration sequence is an i.i.d., then $\operatorname{Var}(X_{n})=O(n^{2})$, $n \to \infty$.
We show this asymptotic bound also remains valid in the case of strictly stationary immigration. For this aim, we recall that $S_{n}=\sum_{i=1}^{n}\zeta_{i}$. Then, due to the strict stationarity of $\left\{\varepsilon_{n}\right\}$, we deduce
$$
\mathbf{E}(S_{n}^{2}) = b^{2}n
+
2\sum_{j=1}^{n-1}(n-j)\rho(j).
$$
According to Theorem A.2 with $p=2$,
$$
\left\|\max_{1\le i\le n}|S_i|\right\|_{2}
\le
C \sqrt{n}\delta_{n,2},
\qquad n \geq 1,
$$
for some constant $C>0$. Hence,
$$
\|S_{n}\|_{2}\le \left\|\max_{1\le i\le n}|S_i|\right\|_{2} \le C \sqrt{n}\delta_{n,2},
\qquad n \geq 1,
$$
which yields
$$
\mathbf{E}(S_{n}^{2})=O(n\delta_{n,2}^{2}), \ \ n\to \infty.
$$
Taking into account $\delta_{n,2}=o\left(\sqrt{n}\right)$, we obtain $\operatorname{Var}(X_{n})=O(n^{2})$.

Next, it is clear that
$$
\mathbf{E}(X_{n}^{2})
=
\operatorname{Var}(X_{n})+\bigl(\mathbf{E}(X_{n})\bigr)^{2}
=
O(n^{2}),
\qquad n \to \infty.
$$
Moreover, from \eqref{eq10} and combining $\mathbf{E}\nu_{n}\leq b^{2}$ with $\mathbf{E}X_{n}=O(n)$, we establish
$$
\mathbf{E}(M_{n}^{2})=O(n), \ \ n\to \infty.
$$
Finally, by the Cauchy--Schwarz inequality,
$$
\mathbf{E}(|M_{n}|)
\le
\sqrt{\mathbf{E}(M_{n}^{2})}
=
O(n^{1/2}), \ \ n\to \infty.
$$
This completes the proof.
\end{proof}

 %Use {thm} for Theorems, {cor} for Corollaries,
 %{lem} for Lemmata, {prop} for Propositions, etc.
 %numbered within sections.

\section{Proof of the main result}

The proof of the main result is organized as follows. First, using representation \eqref{eq2*}, we demonstrate that the process $Y_{n}^{(1)}$ converges weakly in the Skorokhod topology to a diffusion process. Then, by combining the convergence of $A_{n}(t)$ with the continuous mapping theorem, we derive the weak convergence of $Y_{n}^{(1)}$ to the diffusion process $Y$.

\begin{proposition}
Under the assumptions of Theorem 1,
\begin{equation}\label{eq11}
\bigl\{Y_{n}^{(1)}(t),\, t \ge 0 \bigr\}
\mathop{\to}\limits^{\mathcal{D}}
\bigl\{\widetilde{Y}(t),\, t\ge 0 \bigr\}, \ \ n\to \infty,
\end{equation}
where the limit process $\{\widetilde{Y}(t),\, t \ge 0\}$ is the pathwise unique strong solution of the SDE
\begin{equation}\label{eq12}
\mathrm{d}\widetilde{Y}(t)=\lambda\,\mathrm{d}t+\sigma\sqrt{\widetilde{Y}(t)^{+}}\,\mathrm{d}W(t),
\qquad t \ge 0,
\end{equation}
with initial condition $Y(0)=0$.
\end{proposition}

\begin{proof}
Recall from representation \eqref{eq2*} that
$$
Y_{n}^{\left(1\right)}\left(t\right)=\frac{1}{n} \left(X_{0}+\sum_{k=1}^{\left[nt\right]}M_{k}\right)=\frac{1}{n}X_{[nt]}-A_{n}(t).
$$
To prove \eqref{eq11}, we apply Theorem A.1 from the Appendix with
$U_{k}^{(n)}:=n^{-1}M_k$, $U_0^{(n)}:=n^{-1}X_{0}$, $\mathfrak{F}_{k-1}^{(n)}=\mathfrak{F}_{k-1}$, $k,n \geq 1$,
and with coefficient functions
$$
\beta(t,x):=0,
\ \
\gamma(t,x):=\sigma\sqrt{(x+\lambda t)^{+}},
\ \ t\geq 0,\ \  x\in \mathbb{R}.
$$
Then $\mathcal{U}^{(n)}=Y_{n}^{(1)}$, $n\geq 1$ and $\mathcal{U}=\widetilde{Y}$.
We begin with checking that \eqref{eq12} admits a pathwise unique strong solution.
Let $\left\{\widetilde{Y}^{(x)}(t), t\geq 0\right\}$ be a strong solution of \eqref{eq12} with all initial value
$\widetilde{Y}^{(x)}(0)=x\in\mathbb R$, and then define the process $\left\{Z(t), t\geq 0\right\}$ given by
$$
Z(t):=\widetilde{Y}^{(x)}(t)+\lambda t, \qquad t \ge 0.
$$
Then, applying It\^o's formula, one can verify that the process $\left\{Z(t), t \geq 0\right\}$ satisfies the following SDE:
\begin{equation}\label{eq13}
\mathrm{d}Z(t)=\lambda\,\mathrm{d}t+\sigma\sqrt{Z(t)^{+}}\,\mathrm{d}W(t),
\qquad t \ge 0,
\end{equation}
with initial condition $Z(0)=x$. Conversely, if $\left\{Z^{(x)}(t), t\geq 0\right\}$ is a strong
solution of \eqref{eq13} with initial condition $Z^{(x)}(0)=x$, then again from Ito's formula, the process
$$
\widetilde{Y}(t):=Z^{(x)}(t)-\lambda t,\qquad t \ge 0,
$$
is a strong solution of \eqref{eq13} with initial value $\widetilde{Y}(0)=x$.

Observe that the SDE \eqref{eq13} coincides with the SDE \eqref{eq2}. Hence the SDE \eqref{eq13}, and therefore also SDE \eqref{eq12}, has a pathwise unique strong solution for every initial value.
In addition,
\begin{equation}\label{eq14}
\{\widetilde{Y}(t)+\lambda t, t \geq 0\} \stackrel{\mathcal{D}}{=} \{Y(t), t\geq 0\}.
\end{equation}

We note that $\mathbf{E}\bigl((U_{k}^{(n)})^{2}\bigr)<\infty$ for each
$n \geq 1$ and $k \geq 0$. In fact, Lemma 4 yields
$
\mathbf{E}\bigl((U_{k}^{(n)})^{2}\bigr)
=
n^{-2}\mathbf{E}(M_{k}^{2})<\infty$,
$ n,k\geq 1$,
while the assumption $\mathbf{E}(X_{0}^2)<\infty$ entails
$
\mathbf{E}\bigl((U_{0}^{(n)})^2\bigr)
=
n^{-2}\mathbf{E}(X_{0}^{2})<\infty$,
$n \geq 1$.
Finally, since $U_{0}^{(n)}=n^{-1}X_{0}$, it follows that
$U_{0}^{(n)} \to 0$ a.s. as $n\to\infty$,
which yields $U_{0}^{(n)} \xrightarrow{\mathcal{D}} 0$ as $n \to \infty$.

We now verify conditions (a)--(c) of Theorem A.1, which can be expressed in the following forms: for each fixed $T>0$ as $n \to \infty$
\begin{equation}\label{eq15}
\sup_{t\in[0,T]}
\left|
\frac{1}{n}\sum_{k=1}^{[nt]}\mathbf{E}(M_k\mid \mathfrak{F}_{k-1})-0
\right|
\xrightarrow{\mathbf{P}} 0,
\end{equation}
\begin{equation}\label{eq16}
\sup_{t\in[0,T]}
\left|
\frac{1}{n^{2}}\sum_{k=1}^{[nt]}\mathbf{E}(M_{k}^{2}\mid \mathfrak{F}_{k-1})
-
\int_{0}^{t} \sigma^{2} \bigl(Y_{n}^{(1)}(s)+\lambda s\bigr)^{+}\,ds
\right|
\xrightarrow{\mathbf{P}} 0,
\end{equation}
\begin{equation}\label{eq17}
\text{for any } \varepsilon > 0, \qquad
L_{n}(\varepsilon):= \frac{1}{n^{2}}\sum_{k=1}^{[nT]}
\mathbf{E}\left(M_{k}^{2} \, \mathbf{1}{\{|M_{k}| > n\varepsilon\}} \Big| \mathfrak{F}_{k-1} \right)
\xrightarrow{\mathbf{P}} 0,
\end{equation}
We begin with verifying condition \eqref{eq15}. Condition \eqref{eq15} holds trivially for each $T>0$. This follows from the fact that
$\left\{M_{k}, k\geq 1\right\}$ is a martingale difference sequence with respect to
$\left\{\mathfrak{F}_{k}, k\geq 0\right\}$ so that
$\mathbf{E}(M_{k}\mid\mathfrak{F}_{k-1})=0$, $k \ge 1$.

Now we turn to the proof of \eqref{eq16}. Observe that
$$
Y_{n}^{(1)}(s)+A_{n}(s)=\frac{X_{[ns]}}{n} \ge 0,
$$
which yields
$$
\left(Y_{n}^{(1)}(s)+A_{n}(s)\right)^{+} = Y_{n}^{(1)}(s)+A_{n}(s).
$$
Hence
$$
\int_{0}^{t} \sigma^{2} \bigl(Y_{n}^{(1)}+A_{n}(s)\bigr)\,\mathrm{d}s
=\frac{\sigma^2}{n}
\int_{0}^{t} X_{[ns]}\,\mathrm{d}s
=
\frac{\sigma^2}{n^{2}}\sum_{k=0}^{[nt]-1}X_{k}
+
\frac{\sigma^{2}}{n^2}(nt-[nt])X_{[nt]}.
$$
In view of \eqref{eq8}, it follows that
$$
\frac{1}{n^{2}}\sum_{k=1}^{[nt]}\mathbf{E}(M_{k}^{2} \mid \mathfrak{F}_{k-1})
=
\frac{\sigma^{2}}{n^{2}}\sum_{k=1}^{[nt]}X_{k-1}
+
\frac{1}{n^{2}}\sum_{k=1}^{[nt]}\nu_{k}.
$$
%For each $s\geq 0$ and $n \geq 1$, we have
%$$
%M_{n}(u)+\lambda u
%=
%\frac{1}{n}X_{\lfloor nu \rfloor}
%-\left(
%\frac{1}{n}\sum_{i=1}^{[nt]}\eta_{i}-%\lambda u\right).
%$$
Therefore,
$$
\frac{1}{n^{2}}\sum_{k=1}^{[nt]}\mathbf{E}(M_{k}^{2}\mid \mathfrak{F}_{k-1})
-
\int_{0}^{t} \sigma^{2}\bigl(Y_{n}^{(1)}(s)+A_{n}(s)\bigr)\,\mathrm{d}s
$$
$$
=
\frac{1}{n^{2}}\sum_{k=1}^{[nt]}\nu_{k}
-
\frac{\sigma^{2}}{n^{2}}(nt-[nt])X_{[nt]}.
$$
From here, we have an upper bound
$$
\sup_{t \in [0,T]}
\left|
\frac{1}{n^{2}}\sum_{k=1}^{[nt]}\mathbf{E}(M_{k}^{2}\mid \mathfrak{F}_{k-1})
-
\int_{0}^{t} \sigma^{2}\bigl(Y_{n}^{(1)}(s)+A_{n}(s)\bigr)\,\mathrm{d}s
\right|
$$
\begin{equation}\label{eq18}
\le
\frac{1}{n^{2}}\sum_{k=1}^{[nT]}\nu_{k}
+
\frac{\sigma^{2}}{n^{2}}\sup_{t\in[0,T]}X_{[nt]}.
\end{equation}
For the first term of the right-hand side of \eqref{eq18}, note that
$
\mathbf{E}(\nu_{k})\le \mathbf{E}(\varepsilon_{1}^{2})<\infty,
$ which yields
$$
\mathbf{E}\left(\frac{1}{n^{2}}\sum_{k=1}^{[nT]}\nu_{k}\right)
\le
\frac{[nT]}{n^{2}}\mathbf{E}(\varepsilon_{1}^{2}) \to 0, \ \ n\to \infty.
$$
From Chebyshev inequality, we deduce
\begin{equation}\label{eq19}
\frac{1}{n^{2}}\sum_{k=1}^{[nT]}\nu_{k} \xrightarrow{\mathbf P} 0, \ \ n\to \infty.
\end{equation}
For the second term, we use the recursion $X_{i}=X_{i-1}+M_{i}+\eta_{i}$, $i\geq 1$,
and summing them over $i=1$ to $k$, one has that
$$
X_{k}=X_{0}+\sum_{j=1}^{k} M_{j}+A_{k}(1).
$$
Hence, for each $t \geq 0$ and $n \geq 1$,
\begin{equation}\label{eq20}
X_{[nt]}
= \left|X_{[nt]}\right|
\leq X_{0} + \sum_{j=1}^{[nt]} |M_{j}|
+ A_{n}(t).
\end{equation}
Using the obvious bound
$$
\frac{1}{n^{2}}\sup_{t\in[0,T]}
\left|\sum_{j=1}^{[nt]} M_{j}\right|
\leq
\frac{1}{n^{2}}\sum_{j=1}^{[nT]} |M_{j}|
$$
and then Lemma 4, we have
$$
\mathbf{E}\left(
\frac{1}{n^{2}}\sum_{j=1}^{[nT]} |M_{j}|
\right)
=
\frac{1}{n^{2}}\sum_{j=1}^{[nT]} \mathbf{E}M_{j}^{2}
\leq \frac{C}{\sqrt{n}} \to 0,
\qquad n \to \infty.
$$
Due to the Chebyshev inequality, we may claim
\begin{equation}\label{eq21}
\frac{1}{n^{2}}\sup_{t\in[0,T]}
\left|\sum_{j=1}^{[nt]} M_{j}\right| \xrightarrow{\mathbf P} 0.
\end{equation}
Next, by virtue of \eqref{eq5}, it follows
\begin{equation}\label{eq22}
\frac{1}{n^{2}}\sup_{t\in [0,T]}A_{n}(t) \leq \frac{1}{n^{2}}\sup_{t\in [0,T]}\left| A_{n}(t)-\lambda t\right|+ \frac{\lambda T}{n^{2}} \to 0, \ \ n\to \infty.
\end{equation}
Hence, combining \eqref{eq20}, \eqref{eq21} and \eqref{eq22}, we arrive at
\begin{equation}\label{eq23}
\frac{1}{n^{2}}\sup_{t \in [0,T]} X_{[nt]} \xrightarrow{\mathbf P} 0, \ \ n\to \infty.
\end{equation}
Thus, from \eqref{eq18}, \eqref{eq19} and \eqref{eq23}, it follows
\begin{equation}\label{eq24}
\sup_{t\in[0,T]}
\left|
\frac{1}{n^{2}}\sum_{k=1}^{[nt]}\mathbf{E}(M_{k}^{2}\mid \mathfrak{F}_{k-1})
-
\int_{0}^{t} \sigma^{2}\bigl(Y_{n}^{(1)}(s)+A_{n}(s)\bigr)\,\mathrm{d}s
\right|
\xrightarrow{\mathbf P}0.
\end{equation}
On the other hand,
$$
\left|
\int_{0}^{t} \sigma^{2}\bigl(Y_{n}^{(1)}(s)+A_n(s)\bigr)\,\mathrm{d}s
-
\int_{0}^{t} \sigma^{2}\bigl(Y_{n}^{(1)}(s)+\lambda s\bigr)^{+}\,\mathrm{d}s
\right|
\le
\sigma^{2}\int_{0}^{t} |A_{n}(s)-\lambda s|\,\mathrm{d}s,
$$
hence,
$$
\sup_{t\in[0,T]}
\left|
\int_{0}^{t} \sigma^{2}\bigl(Y_{n}^{(1)}(s)+A_n(s)\bigr)\,\mathrm{d}s
-
\int_{0}^{t} \sigma^{2}\bigl(Y_{n}^{(1)}(s)+\lambda s\bigr)^{+}\,\mathrm{d}s
\right|
\le
\sigma^{2} T \sup_{s\in[0,T]}|A_{n}(s)-\lambda s|.
$$
From \eqref{eq5}, we infer
\begin{equation}\label{eq25}
\sup_{t\in[0,T]}
\left|
\int_{0}^{t} \sigma^{2}\bigl(Y_{n}^{(1)}(s)+A_{n}(s)\bigr)\,\mathrm{d}s
-
\int_{0}^{t} \sigma^{2}\bigl(Y_{n}^{(1)}(s)+\lambda s\bigr)^{+}\,\mathrm{d}s
\right|
\xrightarrow{\mathbf{P}}0.
\end{equation}
Finally, \eqref{eq16} follows from \eqref{eq24} and \eqref{eq25}.

Now we check condition \eqref{eq17}. Using \eqref{eq2*} and \eqref{eq3} and the elementary inequality
$ (x+y)^{2} \leq 2(x^{2}+y^{2})$, $x,y \in\mathbb{R}$,
we derive
\begin{equation}\label{eq26}
L_n(\varepsilon)
\leq 2\bigl(L_{n,1}(\varepsilon)+L_{n,2}(\varepsilon)\bigr),
\end{equation}
where
$$
L_{n,1}(\varepsilon):
=
\frac{1}{n^{2}}
\sum_{k=1}^{[nT]}
\mathbf{E}\left(T_{k}^{2}
\mathbf{1}{\{|M_k|>\varepsilon n\}}
\mid \mathfrak{F}_{k-1}
\right),
$$
$$
L_{n,2}(\varepsilon):
=
\frac{1}{n^{2}}
\sum_{k=1}^{[nT]}
\mathbf{E}\left(
N_{k}^{2}
\mathbf{1}{\{|M_{k}|>\varepsilon n\}}
\mid \mathfrak{F}_{k-1}
\right).
$$

We first consider $L_{n,2}(\varepsilon)$. Note that
$L_{n,2}(\varepsilon) \geq 0$ a.s. Observe that
$$
\mathbf{E}L_{n,2}(\varepsilon)
\leq
\frac{4}{n^2}
\sum_{k=1}^{[nT]}
b^{2}\leq \frac{4b^{2}[nT]}{n^{2}}
\to 0, \ \ n\to \infty,
$$
which implies
\begin{equation}\label{eq27}
L_{n,2}(\varepsilon)
\xrightarrow{\mathbf{P}} 0,
\qquad n\to\infty.
\end{equation}
Now consider $L_{n,1}(\varepsilon)$. Applying the obvious inequality
$$
\mathbf{1}{\{|\xi+\eta|> \varepsilon\}}
\leq
\mathbf{1}{\{|\xi|> \varepsilon/2\}}
+
\mathbf{1}{\{|\eta|>\varepsilon/2\}},
$$
which holds for any random variables $\xi$ and $\eta$ and any
$\varepsilon>0$, and considering (7), we have
\begin{equation}\label{eq28}
L_{n,1}(\varepsilon)
\leq
L_{n,1}^{(1)}(\varepsilon)
+
L_{n,1}^{(2)}(\varepsilon),
\end{equation}
where
$$
L_{n,1}^{(1)}(\varepsilon):
=
\frac{1}{n^2}
\sum_{k=1}^{[nT]}
\mathbf{E}\left(
T_k^{2}
\mathbf{1}{\left\{|T_k|>\varepsilon n/2\right\}}
\mid \mathfrak{F}_{k-1}
\right),
$$
$$
L_{n,1}^{(2)}(\varepsilon):
=
\frac{1}{n^{2}}
\sum_{k=1}^{[nT]}
\mathbf{E}\left(
T_k^{2}
\mathbf{1}{\left\{|N_{k}|> \varepsilon n/2\right\}}
\mid \mathfrak{F}_{k-1}
\right).
$$
Conditionally on $\mathfrak{F}_{k-1}$, the random variable $T_k$ is a centered sum of
$X_{k-1}$ i.i.d. copies of $\xi_{1,1}-1$. Therefore, the same argument as in the proof
of the asymptotic negligibility of $L_{n,1}(\varepsilon)$ in \cite{BBP21} remains applicable in the
dependent immigration case. Indeed, the proof argument relies only on the i.i.d. structure of the
offspring variables and on the $\mathfrak{F}_{k-1}$-measurability of $X_{k-1}$ with bounds from Lemma 4.

Hence, for any $\varepsilon>0$,
\begin{equation}\label{eq29}
L_{n,1}^{(1)}(\varepsilon) \xrightarrow[]{\mathbf{P}}0, \ \ n\to \infty.
\end{equation}
Now consider $L_{n,1}^{(2)}(\varepsilon)$.
We now establish that
\begin{equation}\label{eq30}
L_{n,1}^{(2)}(\varepsilon) \xrightarrow{\mathbf{P}} 0,
\qquad n\to\infty.
\end{equation}
Since the families $\{\xi_{k,i}\}$ and $\{\varepsilon_{k}\}$ are mutually
independent, then the random variables
$T_{k}^{2}$
and $\mathbf 1{\{|N_{k}|>\varepsilon n/2\}}$
are conditionally independent with respect to $\mathfrak F_{k-1}$.

Therefore,
$$
\mathbf{E}\!\left(
T_{k}^{2}\mathbf{1} {\{|N_{k}|>\varepsilon n/2\}}
\,\middle|\,\mathfrak{F}_{k-1}
\right)
=
\mathbf{E}(T_{k}^{2}\mid \mathfrak{F}_{k-1})
\mathbf E\!\left(
\mathbf 1{\{|N_{k}|>\varepsilon n/2\}}
\,\middle|\,\mathfrak{F}_{k-1}
\right).
$$
Using the identity
$
\mathbf{E}(T_{k}^{2}\mid \mathfrak{F}_{k-1})=\sigma^{2}X_{k-1}
$, we get
$$
\mathbf{E}\!\left(
T_{k}^{2} \mathbf{1}{\{|N_{k}|>\varepsilon n/2\}}
\,\middle|\,\mathfrak{F}_{k-1}
\right)
=
\sigma^{2}X_{k-1}\,
\mathbf{P}\!\left(|N_{k}|>\varepsilon n/2\,\middle|\,\mathfrak{F}_{k-1}\right).
$$
Hence, applying the Cauchy--Schwarz inequality, we deduce
$$
\mathbf{E} L_{n,1}^{(2)}(\varepsilon)
=
\frac{\sigma^{2}}{n^2}\sum_{k=1}^{[nT]}
\mathbf{E}\!\left(
X_{k-1}\,
\mathbf{P}\!\left(|N_{k}|>\varepsilon n/2\,\middle|\,\mathfrak F_{k-1}\right)
\right).
$$
$$
\leq \frac{\sigma^{2}}{n^{2}}\sum_{k=1}^{[nT]}
\bigl(\mathbf{E} X_{k-1}^{2}\bigr)^{1/2}
\left(
\mathbf{P}\!\left(|N_{k}|>\varepsilon n/2\right)
\right)^{1/2}.
$$
Taking into account that the variables $N_{k}, k\geq 1$ are stationary, we have
$$
\mathbf{E}L_{n,1}^{(2)}(\varepsilon)
\le
\frac{\sigma^{2}}{n^2}
\left(
\mathbf{P}\!\left(|N_{1}|>\varepsilon n/2\right)
\right)^{1/2}
\sum_{k=1}^{[nT]}
\bigl(\mathbf{E} X_{k-1}^{2}\bigr)^{1/2}.
$$
It is clear that
\begin{equation}\label{eq31}
\mathbf{P}\!\left(|N_{1}| > \varepsilon n/2 \right) \to 0, \ \ n\to \infty.
\end{equation}
From Lemma 4, we know that $\mathbf{E} X_{k-1}^{2}=O(k^{2})$, hence,
$$
\sum_{k=1}^{[nT]}
\bigl(\mathbf{E} X_{k-1}^{2}\bigr)^{1/2}=O(n^{2}), \ \ n\to \infty,
$$
and in view of \eqref{eq31},
$$
\mathbf{E} L_{n,1}^{(2)}(\varepsilon)
\le
C
\left(
\mathbf P\!\left(|N_{1}|>\varepsilon n/2\right)
\right)^{1/2} \to 0, \ \ n\to \infty
$$
which ends the proof of \eqref{eq30}. Collecting \eqref{eq28} and \eqref{eq29} together with \eqref{eq30}, we arrive at
\begin{equation}\label{eq32}
L_{n,1}(\varepsilon) \xrightarrow{\mathbf{P}} 0, \ \ n\to \infty.
\end{equation}
Overall, combining \eqref{eq26} and \eqref{eq27} with \eqref{eq32}, we verify the validity of \eqref{eq17}.
\end{proof}

Denote
$Y_{n}^{(2)}:=\{Y_{n}^{(2)}(t)\}$, $Y_{n}^{(3)}:=\{Y_{n}^{(3)}(t)\}$, $t \in [0,T]$.
Furthermore, let
$
\mathbf{0}=\{0,\ t\in[0,T]\}
$
be the process with zero trajectories on $[0,T]$. Without loss of
generality, we assume that these processes are defined on the same probability
space.

It is worth recalling that convergence in distribution of stochastic processes
in the space $D[0,\infty)$ is equivalent to convergence in distribution in
the space $D[0,T]$ for every $T>0$; see Theorem 16.7 in \cite{B99}. This property will used in the subsequent analysis. Furthermore, we denote by $d_{\infty}^{o}$ the metric on
$D[0,\infty)$ which defined by (16.4) in \cite{B99}.

\begin{proposition}
Under the assumptions of Theorem~1, it holds that
\begin{equation}\label{eq33}
Y_{n}^{(2)} \xrightarrow{\mathcal{D}} \mathbf{0},
\qquad n \to \infty.
\end{equation}
\end{proposition}
\begin{proof}
First, we establish that
\begin{equation}\label{eq34}
\sup_{t\in \left[0,T\right]} \left| Y_{n}^{\left(2\right)}\left(t\right) \right|\mathop{\to}\limits^{\mathbf{L}^{2}} 0, \ \ n \to \infty.
\end{equation}
Observe that the sequence $\left\{N_{k}, k \geq 1\right\}$ is a sequence of strictly stationary martingale differences with respect to the filtration $\mathfrak{F}_{k}$. Hence by Doob's maximal $\mathbf{L}_{p}$-inequality it suffices to show that
$$ \frac{1}{n}\left|\sum_{k=1}^{\left[nT \right]}{N_{k}} \right| \mathop{\to}\limits^{\mathbf{L}^{2}} 0, \ \ n\to \infty. $$
Clearly, one can suppose that $T=1$. The application of Burkholder inequality for martingales and taking into account Jensen's inequality together with stationarity of $N_{k}$, gives us
$$ \frac{1}{n^{2}}\mathbf{E}\left(\sum_{k=1}^{n}{N_{k}} \right)^{2} \leq \frac{1}{n^{2}}\sum_{k=1}^{n}\mathbf{E}{N_{k}^{2}} \leq \frac{4}{n}b^{2} \to 0, \ \ n\to \infty $$
which ends the proof \eqref{eq34}.

To complete the proof, it remains to recall the well-known (see \cite{B99}) fact that for any
sequence of elements of $D[0,\infty)$, the uniform convergence coincides with Skorokhod convergence when the limiting
function of that sequence is continuous on $[0,\infty)$. Hence,
$$
d_{\infty}^{o}\bigl(Y_{n}^{(2)},\mathbf{0}\bigr)
=
\sup_{t \in [0,T]}
\frac{1}{n}
\left|
\sum_{k=1}^{[nt]}N_{k}
\right|.
$$
From \eqref{eq34}, it follows that
$$
Y_{n}^{(2)} \xrightarrow{\mathbf{P}} \mathbf{0}, \ \ n\to \infty,
$$
which is equivalent to \eqref{eq33}. This ends the proof.
\end{proof}

\begin{proposition}
Under the assumptions of Theorem~1, it holds that
$$
Y_{n}^{(3)} \xrightarrow{\mathcal{D}} \mathbf{0},
\qquad n \to \infty.
$$
\end{proposition}

\begin{proof}
We first show that
\begin{equation}\label{eq35}
 \sup_{t\in \left[0,T\right]} \left| Y_{n}^{\left(3\right)}\left(t\right)  \right| \mathop{\to}\limits^{\mathbf{L}^{2}} 0, \ \ n \to \infty.
\end{equation}
Applying Theorem A.2 to the $\zeta_{k}$s, we have
$$ \left\| \frac{1}{n}\max_{1\leq i\leq \left[nT\right]}\left|\sum_{k=1}^{i}\zeta_{k}\right|\right\|_{2} \leq
\frac{C}{\sqrt{n}}\left(\left\|\zeta_{1} \right\|_{2}+80 \delta_{n,2}\right)  \to 0, \ \ n\to \infty. $$
Hence, for any fixed $T>0$, as $n \to \infty$,
$$
d_{\infty}^{o}\bigl(Y_{n}^{(3)}, \mathbf{0}\bigr)
=
\sup_{t\in[0,T]}
\frac{1}{n}
\left|
\sum_{k=1}^{[nt]}\zeta_{k}
\right|
\xrightarrow{\mathbf{L}^2}0.
$$
Using the same arguments as in the proof of Proposition~2,
we get the desired convergence.
\end{proof}

Now, relying on the auxiliary results above, one can verify the validity of
Theorem~1.

\begin{proof}[Proof of Theorem~1]
To conclude Theorem 1, it remains to establish \eqref{eq1*}. Recall that
$$
Y_{n}(t)=Y_{n}^{(1)}(t)+A_{n}(t),
\qquad t \ge 0.
$$
Since the the process $A_n$ is random in the dependent immigration case, thus one can not
treat $Y_n$ as the image of $Y_n^{(1)}$ under a deterministic mapping on
$D[0,\infty)$ as in the independent case. Therefore, instead of working with a random
mapping, we consider the deterministic map
$\Phi: D[0,\infty) \times D[0,\infty) \to D[0,\infty)$,
defined by
$$
\Phi(f,g):=f+g.
$$
From \eqref{eq11} and \eqref{eq5}, we know that
$$
Y_{n}^{(1)} \xrightarrow[]{\mathcal{D}}\widetilde Y,
 \ \
A_{n}\xrightarrow[]{\mathbf{P}} A(\cdot),
$$
where $A(t)=\lambda t$. Then, applying Slutsky's theorem, we derive
$$
(Y_{n}^{(1)}, A_n)\xrightarrow[]{\mathcal{D}}(\widetilde Y, A(\cdot)),\ \ n\to \infty,
$$
in $D[0,\infty)\times D[0,\infty)$.
Since $A(\cdot)$ is continuous then the mapping $\Phi$ is continuous at
$(\widetilde Y, A(\cdot))$ a.s. Hence, by continuous mapping theorem,
$$
Y_{n}=\Phi(Y_{n}^{(1)}, A_n)\xrightarrow[]{\mathcal{D}}\Phi(\widetilde Y, A(\cdot)), \ \ n\to \infty.
$$
Consequently,
$$
Y_{n} \xrightarrow[]{\mathcal{D}}Y, \ \ Y(t)=\widetilde Y(t)+\lambda t,
$$
By It\^o's formula, the process $Y$ is the pathwise unique strong solution of
$$
\mathrm{d}Y(t)=\lambda\,dt+\sigma\sqrt{Y(t)^{+}}\,\mathrm{d}W(t),
\qquad Y(0)=0.
$$
Consequently, Theorem~\ref{thm1} follows from a direct application of Propositions~2--3 using Slutsky's theorem.
Thus, Theorem~\ref{thm1} is completely proved.
\end{proof}

\section*{Appendix}
In the proofs we need the following result about convergence of random step
processes towards a diffusion process.
\setcounter{theorem}{0}
\renewcommand{\thetheorem}{A.\arabic{theorem}}
\begin{theorem}[Corollary 2.2 in~\cite{IP10}]
Let $\beta: [0,\infty)\times \mathbb{R} \to \mathbb{R}$ and $\gamma: [0,\infty)\times \mathbb{R} \to \mathbb{R}$ be continuous functions.
Assume that uniqueness in the sense of probability law holds for the SDE
\begin{equation}\label{eq35}
\mathrm{d}\mathcal{U}_{t}=\beta\left(t, \mathcal{U}_{t}\right)\mathrm{d}t+\gamma\left(t, \mathcal{U}_{t}\right)\mathrm{d}W\left(t\right), \ \ t \geq 0,
\end{equation}
with initial condition $\mathcal{U}_{0}=u_{0}$ for all $u_{0}\in \mathbb{R}$, where $W\left(t\right)$ is a standard Wiener
process. Let $\left\{\mathcal{U}_{t}, t \geq 0 \right\}$ be a solution of \eqref{eq35} with initial value $\mathcal{U}_{0}=0$.

For each $n \geq 1$, let $\left\{U_{k}^{\left(n\right)} \right\}$ be a sequence of real-valued random variables adapted to
a filtration $\left\{\mathfrak{F}_{k}^{\left(n\right)}, k \geq 0 \right\}$ such that $\mathbf{E}\left(U_{k}^{\left(n\right)} \right)^{2}<\infty$ for each $n,k \geq 1$. Let
$$ \mathcal{U}_{t}^{\left(n\right)}=\sum_{k=0}^{\left[nt \right]}U_{k}^{\left(n\right)}, \ \ t \geq 0,\ \ n \geq 1. $$
Suppose that as $n\to \infty$,
\begin{align*}
\text{(a)}\quad
&\sup_{t\in [0,T]} \left| \sum_{k=1}^{[nt]} \mathbf{E}\Bigl(U_{k}^{(n)} \big|\ \mathfrak{F}_{k-1}^{(n)} \Bigr)
- \int_{0}^{t} \beta\bigl(s, \mathcal{U}_{s}^{(n)}\bigr) \, \mathrm{d}s \right|
\xrightarrow{\mathbf{P}} 0, \\[6pt]
\text{(b)}\quad
&\sup_{t\in [0,T]} \left| \sum_{k=1}^{[nt]} \operatorname{Var}\Bigl(U_k^{(n)} \big|\ \mathfrak{F}_{k-1}^{(n)} \Bigr)
- \int_{0}^{t} \bigl(\gamma\bigl(s, \mathcal{U}_s^{(n)}\bigr)\bigr)^{2} \, \mathrm{d}s \right|
\xrightarrow{\mathbf{P}} 0, \\[8pt]
\text{(c)}\quad
&\sum_{k=1}^{[nT]} \mathbf{E} \Biggl((U_{k}^{(n)})^{2} \mathbf{1}{\{|U_{k}^{(n)}| > \varepsilon\}}
\Bigg|\ \mathfrak{F}_{k-1}^{(n)} \Biggr)
\xrightarrow{\mathbf{P}} 0.
\end{align*}
Then
$$ \mathcal{U}^{\left(n\right)} \mathop{\to}\limits^{\mathcal{D}} \mathcal{U}, \ \ n\to \infty. $$
\end{theorem}

Let
$\mathbb{T}:\Omega\to\Omega$ be a bijective bi-measurable transformation preserving
the probability. Let $\mathcal{F}_{0}$ be a $\sigma$-algebra of $\mathcal{A}$
satisfying
$
\mathcal{F}_{0}\subseteq \mathbb{T}^{-1}(\mathcal{F}_{0})
$,
and define the nondecreasing filtration $\left\{\mathcal{F}_{i}, i\in\mathbb{Z} \right\}$ by
$\mathcal{F}_{i}=\mathbb{T}^{-i}(\mathcal{F}_{0})$.
Let $\gamma_{0}$ be an $\mathcal{F}_{0}$-measurable centered real random variable.
Define the strictly stationary sequence $\left\{\gamma_{i}, i \in \mathbb{Z}\right\}$ by
$\gamma_{i}=\gamma_{0}\circ \mathbb{T}^{i}$,
and let $S_{n}=\sum_{k=1}^{n}\gamma_{k}$.

In our proofs, the key ingredient is the following sharp Burkholder-type maximal inequality in $\mathbf{L}_{p}$ for a class of stationary sequences that includes martingale sequences, mixingales, and other dependent structures.
\setcounter{theorem}{1}
\renewcommand{\thetheorem}{A.\arabic{theorem}}
\begin{theorem}[Theorem~1 in~\cite{PU05}]
Assume that $\mathbf{E}(|\gamma_{1}|^{p})<\infty$ for some $p\geq 2$. Then, for all $n\geq 1$,
$$
\left\| \max_{1\leq i\leq n}|S_{i}|\right\|_{p}
\leq C_{p}^{1/p} n^{1/2}
\left(\|\gamma_{1}\|_{p}+80\delta_{n,p} \right).
$$
\end{theorem}
We also need the following version of the continuous mapping theorem.
\begin{applemma}[Theorem~3.27 in~\cite{K97}]
Let $(\mathcal{X}, d_\mathcal{X})$ and $(\mathcal{Y}, d_\mathcal{Y})$ be metric spaces, and let
$\left\{\xi_{n}, n\geq 1\right\}$ and $\xi$ be random elements with values in $\mathcal{X}$
such that
$
\xi_n \xrightarrow{\mathcal{D}} \xi$ as $n\to\infty$.
Let $f:\mathcal{X} \to \mathcal{Y}$ and $f_{n}: \mathcal{X}\to \mathcal{Y}$, $n \geq 1$, be measurable mappings,
and let $C\in\mathcal{B}(\mathcal{X})$ be such that
$
\mathbf{P}(\xi\in C)=1.
$
Assume that
$$
\lim_{n\to\infty} d_\mathcal{Y}\bigl(f_{n}(s_{n}), f(s) \bigr)=0
$$
whenever
$$
\lim_{n\to\infty} d_\mathcal{X}(s_{n},s)=0, \qquad s\in C,\quad s_{n} \in \mathcal{X},\ n \geq 1.
$$
Then
$$
f_{n}(\xi_{n})\xrightarrow{\mathcal{D}} f(\xi), \ \  n \to \infty.
$$
\end{applemma}

%We prove fulfilment of condition (27) for the sum in (24) with $U_{k}^{\left(n\right)}=M_{k}/n$ and $\mathfrak{F}_{k}^{\left(n\right)}$ for all $n \geq 1$.
%If your paper includes appendices, then precede the first of them by the command
\appendix
%and then carry on using the \section and \subsection commands, as above.

%%%%%%%%%%%%%%%%%%%%%%%%%%%%%%%%%%%%%%%%%%%%%%%%%%%%%%%%%%%%%%%%%%%%%%%%%%%%%%%%%%%%%%%%%%%%%%%%%%%%%%%%%%
\end{document}